\documentclass[a4paper, 12pt]{article}
\usepackage[margin=1in]{geometry}

\usepackage{caption}

\usepackage{amsmath, amssymb, amsfonts, amsthm}
\numberwithin{equation}{section}

\usepackage{enumitem}

\usepackage[textsize=tiny]{todonotes}
\usepackage{easyReview}

\usepackage{pgf}
\usepackage{tikz}
\usetikzlibrary{arrows,automata}
\usetikzlibrary{positioning}
\usetikzlibrary{decorations.pathmorphing}
\usetikzlibrary{shapes.geometric}
\usetikzlibrary{fit}
\usetikzlibrary{backgrounds}

\newcommand{\mc}{\mathcal}
\newcommand{\mbb}{\mathbb}

\usepackage{mathtools}
\DeclarePairedDelimiter\abs{\lvert}{\rvert}

\newtheorem{lem}{Lemma}[section]
\newtheorem{thm}[lem]{Theorem}
\newtheorem{cor}[lem]{Corollary}
\newtheorem{prop}[lem]{Proposition}

\theoremstyle{definition}

\newtheorem{ex}[lem]{Example}
\theoremstyle{remark}
\newtheorem{rem}[lem]{Remark}

\usepackage{array}
\newcolumntype{M}[1]{>{\centering\arraybackslash}m{#1}}
\newcolumntype{N}{@{}m{0pt}@{}}

\usepackage{hyperref}

\title{$\varepsilon$-Almost collision-flat universal hash functions and mosaics of designs\footnote{This work was presented in part at the 2023 IEEE International Symposium on Information Theory and at the 29th Nordic Congress of Mathematicians. A short and preliminary version of this paper is \cite{ISIT23}. It mainly contains: two of the lower bounds on the size of an ACFU hash function (Theorem 3.4) without proof; a less general version of Theorem 4.1, which shows how to construct an ACFU hash function from an almost universal hash function; all examples without the discussion of their design-theoretic properties; a discussion of the application of ACFU hash functions to privacy amplification which is more detailed than here (Section 5) and gives a better bound on the attained security level.}}
\author{Moritz Wiese\thanks{Technical University of Munich, Chair of Theoretical Information Technology and BMBF Research Hub 6G-Life, Arcisstra\ss e 21, 80333 M\"unchen, Germany. Corresponding author. E-mail: wiese@tum.de}
\and Holger Boche\thanks{Technical University of Munich, Chair of Theoretical Information Technology and BMBF Research Hub 6G-Life, Arcisstraße 21, 80333 M\"unchen, Germany. E-mail: boche@tum.de}}

\begin{document}
    
\maketitle 

\abstract{
    We introduce, motivate and study $\varepsilon$-almost collision-flat universal (ACFU) hash functions $f:\mc X\times\mc S\to\mc A$. Their main property is that the number of collisions in any given value is bounded. Each $\varepsilon$-ACFU hash function is an $\varepsilon$-almost universal (AU) hash function, and every $\varepsilon$-almost strongly universal (ASU) hash function is an $\varepsilon$-ACFU hash function. We study how the size of the seed set $\mc S$ depends on $\varepsilon,\abs{\mc X}$ and $\abs{\mc A}$. Depending on how these parameters are interrelated, seed-minimizing ACFU hash functions are equivalent to mosaics of balanced incomplete block designs (BIBDs) or to duals of mosaics of quasi-symmetric block designs; in a third case, mosaics of transversal designs and nets yield seed-optimal ACFU hash functions, but a full characterization is missing. By either extending $\mc S$ or $\mc X$, it is possible to obtain an $\varepsilon$-ACFU hash function from an $\varepsilon$-AU hash function or an $\varepsilon$-ASU hash function, generalizing the construction of mosaics of designs from a given resolvable design (Gnilke, Greferath, Pav\v cevi\'c, Des.\ Codes Cryptogr.~86(1)). The concatenation of an ASU and an ACFU hash function again yields an ACFU hash function. Finally, we motivate ACFU hash functions by their applicability in privacy amplification.
}
\vspace{1em}


\section{Introduction}

Let $\mc X,\mc S,\mc A$ be finite sets. For $\varepsilon\geq0$, we call a function $f:\mc X\times\mc S\to\mc A$ an \textit{$\varepsilon$-almost collision-flat universal ($\varepsilon$-ACFU) hash function} if 
\begin{enumerate}[label=(ACFU\arabic*), wide=0pt, labelindent=\parindent]
    \item for every $x\in\mc X$ and every $\alpha\in\mc A$,
    \[
        \abs{\{s:f(x,s)=\alpha\}}=\frac{\abs{\mc S}}{\abs{\mc A}},
    \]
    \item for all distinct $x,x'\in\mc X$ and every $\alpha\in\mc A$,
    \[
        \abs{\{s:f(x,s)=f(x',s)=\alpha\}}\leq\frac{\varepsilon\abs{\mc S}}{\abs{\mc A}}.
    \]
\end{enumerate}
We call $f$ \textit{nontrivial} if $2\leq\abs{\mc A}<\abs{\mc X}$. The set $\mc S$ is called the \textit{seed set} of $f$; occasionally, we will call $\mc X$ the \textit{point set} of $f$. In this paper, we are going to motivate $\varepsilon$-ACFU hash functions and study their properties. 

There are two well-known related types of hash functions which we will mention and use frequently. An \textit{$\varepsilon$-almost universal ($\varepsilon$-AU) hash function} $f:\mc X\times\mc S\to\mc A$ satisfies
\begin{enumerate}[label=(AU), wide=0pt, labelindent=\parindent, topsep=0.6em]
    \item for all distinct $x,x'\in\mc X$,
    \[
        \abs{\{s:f(x,s)=f(x',s)\}}\leq\varepsilon\abs{\mc S}.
    \]
\end{enumerate}
This definition goes back to Stinson \cite{Stinson_UH_AC} as a generalization of \textit{universal hash functions} due to Carter and Wegman \cite{CW_UH}, where $\varepsilon=1/\abs{\mc A}$. Finally, an \textit{$\varepsilon$-almost strongly universal ($\varepsilon$-ASU) hash function} $f:\mc X\times\mc S\to\mc A$ satisfies the properties
\begin{enumerate}[label=(ASU\arabic*), wide=0pt, labelindent=\parindent]
    \item for every $x\in\mc X$ and every $\alpha\in\mc A$,
    \[
        \abs{\{s:f(x,s)=\alpha\}}=\frac{\abs{\mc S}}{\abs{\mc A}}
    \]
    (so this is the same as (ACFU1)),
    \item for all distinct $x,x'\in\mc X$ and all $\alpha,\alpha'\in\mc A$,
    \[
        \abs{\{s:f(x,s)=\alpha,f(x',s)=\alpha'\}}\leq\frac{\varepsilon\abs{\mc S}}{\abs{\mc A}}.
    \]
\end{enumerate}
\textit{Strongly universal (SU) hash functions}, where the above properties hold with $\varepsilon=1/\abs{\mc A}$, were defined by Wegman and Carter \cite{WC_UH}, and the above definition of $\varepsilon$-ASU hash functions is again due to Stinson \cite{Stinson_UH_AC}. We also refer to the $\mc S$ sets of AU and ASU hash functions as their seed sets.

\begin{rem}\label{rem:epsilon_convention}
    Sometimes $\varepsilon$-AU hash functions are defined in terms of the parameter $\varepsilon\abs{\mc A}$, e.g., \cite{TsuruHay_DualUniv, Hay_almost_dual_UHF}. Our convention is the same as in \cite{Stinson_UH_AC}. 
\end{rem}

For a hash function $f$, the event that $f(x,s)=f(x',s)$ for distinct $x,x'$ is usually referred to as a \textit{collision}. If $f$ is an AU hash function, it is not important in which value of $f$ a collision occurs, only the total number of collisions has to be bounded. In contrast, for an ACFU hash function, we are also interested in the value of $f$ when a collision occurs. In the range of numbers $\abs{\{s:f(x,s)=f(x',s)=\alpha\}}$, where $x\neq x'$ are fixed and where $\alpha$ ranges over all of $\mc A$, there should be no peaks (valleys are allowed). Hence the name ``collision-flat''.

The following relations between the three types of hash functions are obvious.

\begin{lem}\label{lem:hash_fct_rels}
    Let $\varepsilon\geq0$. 
    \begin{enumerate}
        \item Every $\varepsilon$-ACFU hash function is an $\varepsilon$-AU hash function.
        \item Every $\varepsilon$-ASU hash function is an $\varepsilon$-ACFU hash function.
    \end{enumerate}
\end{lem}

Our motivation for studying ACFU hash functions comes from their applicability in \textit{privacy amplification}. This is a key step in \textit{secret key generation}, where two parties A and B want to generate a shared secret key from correlated random observations and public discussion, while an adversary can observe their public communication and possibly also makes observations correlated to the key-generating parties' ones \cite[Chapter 4]{BlochBarros}. The privacy amplification step takes place once the two parties have agreed on a shared random variable $X$. The result of privacy amplification is a key shared by A and B which is secret with respect to the adversary. We formalize this setting in Section \ref{sect:motivation}. 

The standard choice of function to use in privacy amplification so far has been an arbitrary AU hash function \cite{BBR_PA, BBCM_gen_PA}; seeded extractors are also suitable \cite{MauW_extractors}. These functions need additional randomness, a \textit{seed}, as a second input (in the definitions above, from the set $\mc S$), and they guarantee security against all adversaries for which the entropy of $X$ conditional on their own observation is sufficiently large. It is desirable to make the seed small, since its generation is costly and it must be known by both A and B.

For general AU hash functions and extractors, key security so far has only been proved under the assumption that the adversary has no other information about the key than the knowledge of the underlying joint probability distribution together with its observations. A stronger security test goes as follows: The adversary has the above information, and additionally, it has to distinguish between two arbitrary possible key values. Security (\textit{key indistinguishability}) is declared if the adversary's decision performance is no more than negligibly better than random guessing for all such value pairs. In all cases we know of where this stronger security criterion has been applied and where security is guaranteed against all $X$ with sufficiently high conditional entropy, the AU hash functions which were used for privacy amplification or in a related problem, the wiretap channel problem, actually are ACFU hash functions \cite{HayMat, mosaics}. (These functions will be considered below as examples.) For this reason, ACFU hash functions are worth a closer look. In Section \ref{sect:motivation}, we sketch how key indistinguishability can be proven in a privacy amplification scenario using ACFU hash functions.

As we have said already, the second argument $s$ to an ACFU hash function will be chosen uniformly at random. Since randomness is expensive, our first question about ACFU hash functions (Section \ref{sect:seed_lbs}) is how small the set $\mc S$ can be, given $\varepsilon$ and the cardinalities of $\mc X$ and $\mc A$. We derive three lower bounds, each of which is relevant in different regimes depending on the relation of $\varepsilon,\abs{\mc X},\abs{\mc A}$. We also discuss which structure an ACFU hash function attaining equality in any of the bounds has. 

It turns out that if $\varepsilon$ is optimal (i.e., minimal given $\abs{\mc X}$ and $\abs{\mc A}$), then an $\varepsilon$-ACFU hash function attains equality in the corresponding lower bound if and only if its underlying structure is that of a mosaic of balanced incomplete block designs (BIBDs) \cite{GGP_mosaics}. The two other lower bounds on the seed size hold for general $\varepsilon$. One of them holds for small $\varepsilon$, but may be void if $\abs{\mc A}$ is large relative to $\abs{\mc X}$ (roughly $\abs{\mc A}^2>\abs{\mc X}$). If this bound applies, the ACFU hash functions satisfying equality correspond precisely to the duals of mosaics of quasi-symmetric BIBDs \cite{mosaics}. In the remaining bound, we do not have a full characterization of the ACFU hash functions attaining equality, but we show that equality is attained by mosaics of transversal designs or of nets. Due to these facts, we believe that ACFU hash functions also are interesting objects of study in themselves. In \cite{mosaics}, the authors studied mosaics of balanced incomplete block designs and of group divisible designs and applied them to privacy amplification as well as to another problem from information-theoretic security, the wiretap channel. The results of the present paper embed such mosaics in the wider picture of ACFU hash functions. The necessary design-theoretic definitions are collected in Section \ref{sect:prelims}.

In addition to studying how small $\mc S$ can be for given $\varepsilon$ and $\abs{\mc X},\abs{\mc A}$, we also give three methods for constructing ACFU hash functions in Section \ref{sect:constructions}. By the first one, based on an extension of the seed set with the help of a quasigroup (latin square), one obtains an $\varepsilon$-ACFU hash function from another function if and only if the latter is an $\varepsilon$-AU hash function. This can be seen as a generalization of the method proposed in \cite{GGP_mosaics} for constructing a mosaic of BIBDs from a resolvable BIBD, and also as a generalization of a method used frequently to construct ASU hash functions from AU hash functions. The second method is the dual of the first and produces an $\varepsilon$-ACFU hash function if and only if the original function is an $\varepsilon$-ASU hash function. In the case where the quasigroup in fact is an abelian group, we characterize those ACFU hash functions which can be derived by both methods as ``double extensions'' of $\varepsilon$\textit{-balanced functions}. By the third construction method, one obtains an $\varepsilon$-ACFU hash function by concatenating an $\varepsilon_1$-ACFU hash function and an $\varepsilon_2$-ASU hash function, similar to methods used by Stinson \cite{Stinson_UH_AC} for constructing AU and ASU hash functions from other AU or ASU hash functions. For all these methods, we discuss if and how seed optimality of the original functions can lead to seed optimality of the resulting functions.

For a practical application of ACFU hash functions, it will be important to find examples which have a low computational complexity. All our examples can be computed efficiently, although their complexities differ in concrete detail. We will not consider this issue further; more on this was said in \cite{mosaics} on the examples given there, some of which we will meet later on in this paper (Examples \ref{ex:aff_OCFU}, \ref{ex:Denniston} and \ref{ex:TD}).

\section{Preliminaries on (mosaics of) incidence structures}\label{sect:prelims}

An \textit{incidence structure} is a triple $D=(\mc X,\mc S,I)$, where $\mc X$ and $\mc S$ are finite sets and $I$ is an incidence relation on $\mc X\times\mc S$. $\mc X$ is called the \textit{point set} and $\mc S$ the \textit{block index set} of $D$. We also say that $D$ is an incidence structure \textit{on} $\mc X$. The sets
\[
	\{x\in\mc X:x\text{ incident with }s\},
\]
where $s\in\mc S$, are usually referred to as the \textit{blocks} of $D$. The \textit{dual} of $D$ is the incidence structure $\tilde D$ on $\mc S$ with block index set $\mc X$ and $s$ incident with $x$ in $\tilde D$ if and only if $x$ is incident with $s$ in $D$. 

An incidence structure $D$ is \textit{resolvable} if the block index set can be partitioned into \textit{parallel classes} of constant size in such a way that for each parallel class, a point $x$ is incident with a unique element of the class. (Note that when considering the partition of $\mc X$ into the blocks corresponding to a given parallel class, some of these blocks may be empty.) If $D$ is resolvable, one can index the block indices of each parallel class by a set $\mc A$, and the block index set $\mc S$ of $D$ can be written as a Cartesian product $\mc S=\mc H\times\mc A$, where $\mc H$ is an index set for the parallel classes.

Two incidence structures $D=(\mc X,\mc S,I)$ and $D'=(\mc X',\mc S',I')$ are called \textit{isomorphic} if there exist bijective mappings $\phi:\mc X\to\mc X'$ and $\psi:\mc S\to\mc S'$ such that any $x\in\mc X$ and $s\in\mc S$ are incident in $D$ if and only if $\phi(x)$ is incident with $\psi(s)$ in $D'$.

A \textit{balanced incomplete block design (BIBD)} on a finite set $\mc X$ is an incidence structure $(\mc X,\mc S,I)$ for which there exist positive integers $k$ and $\lambda$ such that
\begin{enumerate}
	\item each $s\in\mc S$ is incident with exactly $k$ elements of $\mc X$,
	\item for any two distinct $x,x'\in\mc X$, there are exactly $\lambda$ elements of $\mc S$ incident with both $x$ and $x'$.
\end{enumerate}
If $\abs{\mc X}=v$, then such a BIBD is called a BIBD$(v,k,\lambda)$. We call a BIBD \textit{nontrivial} if $1<k<v$. 

In a BIBD$(v,k,\lambda)$ with $\abs{\mc S}=b$, a point $x$ is incident with exactly $r$ block indices $s$, where $r$ satisfies
\begin{equation}\label{eq:1-design}
	bk=vr.
\end{equation}
Another important relation is
\begin{equation}\label{eq:lambda_v_r_k}
	\lambda(v-1)=r(k-1).
\end{equation}

A BIBD is \textit{quasi-symmetric} if there are two distinct numbers $\mu_1,\mu_2$ (the \textit{intersection numbers}) such that two blocks intersect in either $\mu_1$ or $\mu_2$ points. By definition, the only BIBDs which are both resolvable and quasi-symmetric are the \textit{affine designs}; they satisfy
\begin{equation}\label{eq:aff_des_blocks}
	b=v+r-1.
\end{equation}
A BIBD is \textit{symmetric} if the intersection of any two distinct blocks has constant size; in this case, the point and the block index sets have the same cardinality.  More details on BIBDs can be found in \cite{BJL_book}, quasi-symmetric designs are treated in depth in \cite{ShrSin_QS_Des}.

Another type of design which will play a role in this paper is the \textit{net} \cite{Dembowski}. It satisfies
\begin{enumerate}
	\item To every point (block index) there exist two block indices (points) not incident with it,
	\item two points are incident with at most one common block index,
	\item if a point $x$ is not incident with a block index $s$, then there exists one and only one block index $s'$ incident with $x$ such that the blocks pertaining to $s$ and $s'$ have empty intersection.
\end{enumerate}
Note that a net is resolvable. A point of a net is incident with a constant number of block indices, which means that nets satisfy \eqref{eq:1-design} as well. By definition, a \textit{transversal design} is the dual of a net \cite{BJL_book}, so its point set can be partitioned into point classes such that two distinct points are joined by a unique block index if and only if they are contained in different point classes. 

Let $\mc A$ be a finite set. A \textit{mosaic of incidence structures on $\mc X$} is a family $M=(D_\alpha)_{\alpha\in\mc A}$ such that for some set $\mc S$,
\begin{enumerate}
	\item each $D_\alpha$ is an incidence structure with point set $\mc X$ and block index set $\mc S$,
	\item every pair $(x,s)\in\mc X\times\mc S$ is incident in a unique $D_\alpha$.
\end{enumerate}
The incidence structures $D_\alpha$ are called the \textit{members} of $M$. We call $M$ \textit{nontrivial} if $2\leq\abs{\mc A}<\abs{\mc X}$.

The \textit{sum} $\Sigma M$ of $M$ is the incidence structure on $\mc X$ with block index set $\mc S\times\mc A$ where $x$ is incident with $(s,\alpha)$ if and only if $x$ is incident with $s$ in $D_\alpha$. The \textit{dual} of $M$ is the mosaic $\tilde M=(\tilde D_\alpha)_{\alpha\in\mc A}$ on $\mc S$ with block index set $\mc X$, where each $\tilde D_\alpha$ is the dual of $D_\alpha$. By a \textit{mosaic of BIBD$(v,k,\lambda)$}, we mean a mosaic each of whose members is a BIBD$(v,k,\lambda)$. Mosaics of BIBDs were defined in \cite{GGP_mosaics}. Mosaics of incidence structures each of whose members is a BIBD, though with different parameters, have appeared earlier in the literature \cite{CKZ_tilings}. Mosaics of group divisible designs and duals thereof were defined in \cite{mosaics}.

The connection of mosaics of incidence structures with functions is the following. Let $f:\mc X\times\mc S\to\mc A$ be any function. Then each inverse image $f^{-1}(\alpha)$ determines an incidence relation $D_\alpha$ on $\mc X$ with block index set $\mc S$, in such a way that $x$ is incident with $s$ if and only if $f(x,s)=\alpha$. Thus every $f:\mc X\times\mc S\to\mc A$ determines a unique mosaic $M(f)$ of incidence structures on $\mc X$. Conversely, every mosaic $M=(D_\alpha)_{\alpha\in\mc A}$ gives rise to a unique function $f_M:\mc X\times\mc S\to\mc A$, where $\mc X$ is the point set and $\mc S$ the block index set of the mosaic. 

We note the following connection of mosaics with resolvability.

\begin{lem}\label{lem:mos_res}
	For any finite set $\mc X$, there is a one-to-one correspondence between the mosaics of incidence structures $M$ on $\mc X$ and the resolvable incidence structures $D$ on $\mc X$ such that $\Sigma M=D$.
\end{lem}

\begin{proof}
    Let $M=(D_\alpha)_{\alpha\in\mc A}$ be a mosaic on $\mc X$ with block index set $\mc S$. For each $s\in\mc S$, the set of pairs $\{(s,\alpha):\alpha\in\mc A\}$ forms a parallel class for $\Sigma M$. Conversely, let $D$ be a resolvable incidence structure on $\mc X$ with block index set $\mc H\times\mc A$, where $\mc H$ is an index set for the parallel classes of $D$ and $\mc A$ is an index set for the elements of each parallel class. For each $\alpha\in\mc A$, define the incidence structure $D_\alpha$ on $\mc X$ with block index set $\mc H$ by letting a point $x$ be incident with block index $h$ in $D_\alpha$ if and only if $x$ is incident with $(h,\alpha)$. Clearly, $(D_\alpha)_{\alpha\in\mc A}$ is a mosaic of incidence structures on $\mc X$.
\end{proof}

\section{Structure of ACFU hash functions}\label{sect:seed_lbs}

In order to simplify our notation when we analyze $\varepsilon$-ACFU hash functions in terms of mosaics of incidence structures, we introduce two types of blocks associated with a function $f:\mc X\times\mc S\to\mc A$. For $x\in\mc X$ and $\alpha\in\mc A$, we define
\begin{equation}\label{eq:primal_blocks}
    B_{x,\alpha}=\{s\in\mc S:f(x,s)=\alpha\}
\end{equation}
and for $s\in\mc S$ and $\alpha\in\mc A$, we set
\begin{equation}\label{eq:dual_blocks}
    B_{s,\alpha}=\{x\in\mc X:f(x,s)=\alpha\}.
\end{equation}
We will call a set of any of these types a block. Note that the blocks $B_{s,\alpha}$ ($s\in\mc S$) are precisely the blocks of the $\alpha$-th member $D_\alpha$ of $M(f)$, whereas the blocks $B_{x,\alpha}$ ($x\in\mc X$) are the blocks of the dual of $D_\alpha$. The function $f$ is an $\varepsilon$-ACFU hash function if the $B_{x,\alpha}$ have constant size and if the intersection $B_{x,\alpha}\cap B_{x',\alpha}$ for distinct $x,x'\in\mc X$ has cardinality at most $\varepsilon\abs{\mc S}/\abs{\mc A}$.

We start our analysis by observing that $\varepsilon$ cannot be arbitrarily small for an $\varepsilon$-ACFU hash function.

\begin{lem}\label{lem:epsilon_lb}
    For any $\varepsilon$-ACFU hash function $f:\mc X\times\mc S\to\mc A$, 
    \begin{equation}\label{eq:eps_lb}
        \varepsilon\geq\frac{\abs{\mc X}-\abs{\mc A}}{\abs{\mc A}(\abs{\mc X}-1)}.
    \end{equation}
    If equality holds, then $\Sigma M(f)$ (the sum of the mosaic determined by $f$, see Section \ref{sect:prelims}) is a resolvable BIBD on $\mc X$.
\end{lem}

If $\abs{\mc X}$ and $\abs{\mc A}$ are given, we call the quantity on the right-hand side of \eqref{eq:eps_lb} the \textit{optimal} $\varepsilon$. The proof of Lemma \ref{lem:epsilon_lb} follows immediately from the first part of Lemma \ref{lem:hash_fct_rels} together with the following results of Sarwate \cite{Sarwate_UnivHashFcts} and Stinson \cite{Stinson_comb_tech_UH}.

\begin{lem}[\cite{Sarwate_UnivHashFcts}, p.~42 and \cite{Stinson_comb_tech_UH}, Theorem 1.1]\label{lem:opt_eps_lem}
	For any $\varepsilon$-AU hash function, $\varepsilon$ satisfies \eqref{eq:eps_lb}.
\end{lem}

\begin{lem}[\cite{Stinson_comb_tech_UH}, Theorem 2.1]\label{lem:Stinson_lemma}
    \begin{enumerate}
    	\item If $f:\mc X\times\mc S\to\mc A$ is an $\varepsilon$-AU hash function where $\varepsilon$ equals the right-hand side of \eqref{eq:eps_lb}, then $\Sigma M(f)$ is a resolvable BIBD.
    	\item Conversely, let $D$ be a resolvable BIBD on $\mc X$. Then, the function $f_M$ determined by the mosaic $M$ corresponding to $D$ by Lemma \ref{lem:mos_res} is an $\varepsilon$-AU hash function with $\varepsilon$ satisfying equality in \eqref{eq:eps_lb}.
    \end{enumerate}
\end{lem}

We call an $\varepsilon$-ACFU hash function with optimal $\varepsilon$ an \textit{optimally collision-flat universal (OCFU)} hash function. An $\varepsilon$-AU hash function with optimal $\varepsilon$ is called \textit{optimally universal (OU)} by Sarwate and Stinson.

\subsection{General $\varepsilon$}

We start the analysis of the structure of ACFU hash functions by giving two bounds on the size of $\mc S$. Then we give criteria for equality and discuss the relation between the bounds.

\begin{thm}\label{thm:seed_size_lb}
    For a nontrivial $\varepsilon$-ACFU hash function $f:\mc X\times\mc S\to\mc A$, it holds that
    \begin{equation}\label{eq:seed_size_lb_variance}
        \abs{\mc S}
        \geq 1+\frac{\abs{\mc X}(\abs{\mc A}-1)^2}{\varepsilon\abs{\mc A}(\abs{\mc X}-\abs{\mc A})+\abs{\mc A}^2-\abs{\mc X}}
    \end{equation}
    and
    \begin{equation}\label{eq:seed_size_lb_simple}
        \abs{\mc S}\geq\frac{\abs{\mc A}}{\varepsilon}.
    \end{equation}
\end{thm}

\begin{proof}
    The bound \eqref{eq:seed_size_lb_simple} is easy to see: If $\varepsilon\abs{\mc S}/\abs{\mc A}<1$, then $\abs{\{s:f(x,s)=f(x',s)=\alpha\}}=0$ for every $\alpha$ and all distinct $x,x'$. In particular, each of the functions $x\mapsto f(x,s)$ would be injective, meaning that $\abs{\mc A}\geq\abs{\mc X}$. But this is impossible since we assume that $f$ is nontrivial.

    To prove \eqref{eq:seed_size_lb_variance}, we use the ``variance method'' also used by Stinson \cite{Stinson_UH_AC} to find lower bounds on $\mc S$ in the case of $\varepsilon$-AU and $\varepsilon$-ASU hash functions. Fix an arbitrary $\alpha\in\mc A$ and $s\in\mc S$ and define $\lambda_\sigma=\abs{B_{s,\alpha}\cap B_{\sigma,\alpha}}$ (recall the notation \eqref{eq:primal_blocks}). Then obviously,
    \[
        \sum_{\sigma\neq s}1=\abs{\mc S}-1.
    \]
    Note that $\abs{\mc S}-1\geq 1$ since we can always assume $\varepsilon\leq 1$, so that \eqref{eq:seed_size_lb_simple} implies $\abs{\mc S}\geq\abs{\mc A}\geq 2$.     Also, counting pairs $(x,\sigma)$ such that $\sigma\neq s$ and $x\in B_{s,\alpha}\cap B_{\sigma,\alpha}$, and using property (ACFU1), we find that
    \begin{align*}
        \sum_{\sigma\neq s}\lambda_\sigma
        =\sum_{x:f(x,s)=\alpha}\abs{\{\sigma\neq s:f(x,\sigma)=\alpha\}}
        &=\abs{B_{s,\alpha}}\left(\frac{\abs{\mc S}}{\abs{\mc A}}-1\right).
    \end{align*}
    Moreover, counting triples $(\sigma,x,x')$ with $\sigma\neq s$ as well as $x\neq x'$ and $x,x'\in B_{s,\alpha}\cap B_{\sigma,\alpha}$, we find using property (ACFU2) that
    \begin{align*}
        &\sum_{\sigma\neq s}\lambda_\sigma(\lambda_\sigma-1)\\
        &=\sum_{x:f(x,s)=\alpha}\sum_{x'\neq x:f(x',s)=\alpha}\abs{\{\sigma\neq s:f(x,\sigma)=f(x',\sigma)=\alpha\}}\\
        &\leq\left(\frac{\varepsilon\abs{\mc S}}{\abs{\mc A}}-1\right)\abs{B_{s,\alpha}}(\abs{B_{s,\alpha}}-1).
    \end{align*}
    The mean of the $\lambda_\sigma$ is
    \[
        \overline\lambda=\frac{\abs{B_{s,\alpha}}\left(\frac{\abs{\mc S}}{\abs{\mc A}}-1\right)}{\abs{\mc S}-1}.
    \]
    Hence
    \begin{align*}
        0
        &\leq \sum_{\sigma\neq s}(\lambda_\sigma-\overline\lambda)^2\\
        &=\sum_{\sigma\neq s}\lambda_\sigma^2-(\abs{\mc S}-1)\overline\lambda^2\\
        &\leq\left(\frac{\varepsilon\abs{\mc S}}{\abs{\mc A}}-1\right)\abs{B_{s,\alpha}}(\abs{B_{s,\alpha}}-1)+\abs{B_{s,\alpha}}\left(\frac{\abs{\mc S}}{\abs{\mc A}}-1\right)-\frac{\abs{B_{s,\alpha}}^2\left(\frac{\abs{\mc S}}{\abs{\mc A}}-1\right)^2}{\abs{\mc S}-1}.
    \end{align*}
    Solving this for $\abs{\mc S}$, one sees that it implies that 
    \begin{equation}\label{eq:lb_raw}
        \abs{\mc S}
        \geq 1+\frac{(\abs{\mc A}-1)^2\abs{B_{s,\alpha}}}{\varepsilon\abs{\mc A}(\abs{B_{s,\alpha}}-1)+\abs{\mc A}-\abs{B_{s,\alpha}}}.
    \end{equation}
    In order to make sure that the denominator does not vanish, we observe that if this were the case, we would have
    \[
    	\varepsilon
    	=\frac{\abs{B_{s,\alpha}}-\abs{\mc A}}{\abs{\mc A}(\abs{B_{s,\alpha}}-1)},
    \]
    which in the light of Lemma \ref{lem:epsilon_lb} would mean that $\abs{B_{s,\alpha}}=\abs{\mc X}$. But then $\varepsilon$ would satisfy equality in \eqref{eq:eps_lb} and $f$ would be an OU hash function, so $\abs{B_{s,\alpha}}=\abs{\mc X}/\abs{\mc A}$. However since $f$ is nontrivial, it is impossible that $\abs{\mc A}=1$. Hence the right-hand side of \eqref{eq:lb_raw} is well-defined.
    
    We obtain the claimed bound by observing that for every $s\in\mc S$, there must be a value $\alpha$ such that $\abs{B_{s,\alpha}}\geq\abs{\mc X}/\abs{\mc A}$, and inserting this in \eqref{eq:lb_raw}.
\end{proof}

The proof of \eqref{eq:seed_size_lb_variance} provides us with enough information to precisely characterize the structure of those ACFU hash functions which attain equality.

\begin{cor}\label{cor:eq_bed_variance}
    If $f:\mc X\times\mc S\to\mc A$ is a nontrivial $\varepsilon$-ACFU hash function satisfying equality in \eqref{eq:seed_size_lb_variance}, then the dual of its mosaic $M(f)$ is a mosaic of quasi-symmetric BIBD$(v,k,\lambda)$, where
    \[
        v=\abs{\mc S},\quad k=\frac{\abs{\mc S}}{\abs{\mc A}},\quad\lambda=\frac{\abs{\mc X}(\abs{\mc S}-\abs{\mc A})}{\abs{\mc A}^2(\abs{\mc S}-1)},
    \]
    and the intersection numbers are $0$ and $\varepsilon\abs{\mc S}/\abs{\mc A}$. 
    
    Conversely, if $M$ is the dual of a mosaic of nontrivial quasi-symmetric BIBD$(v,k,\lambda)$ with intersection numbers $0$ and $\mu$, then $f_M$ is a nontrivial $\varepsilon$-ACFU hash function with
    \[
    	\varepsilon=\frac{\mu}{k}.
    \]
\end{cor}

\begin{proof}
    Let $f:\mc X\times\mc S\to\mc A$ be an $\varepsilon$-ACFU hash function. Equality in \eqref{eq:seed_size_lb_variance} holds if and only if in the proof of \eqref{eq:seed_size_lb_variance}, all inequalities are equalities irrespective of the choice of $s$ and $\alpha$. Hence, equality holds if and only if the following conditions are satisfied:
    \begin{enumerate}
        \item $\abs{B_{s,\alpha}}$ is constant and equal to $\abs{\mc X}/\abs{\mc A}$;
        \item for all distinct $s,s'\in\sigma$ and all $\alpha$,
        \[
            \abs{B_{s,\alpha}\cap B_{s',\alpha}}=\overline\lambda
        \]
        for $\overline\lambda$ as in the proof of Theorem \ref{thm:seed_size_lb}; another way of stating this is that any two distinct $s$ and $s'$ are contained in precisely $\overline\lambda$ different blocks of the form $B_{x,\alpha}$;
        \item for any distinct $x,x'$ for which the corresponding blocks $B_{x,\alpha}$ and $B_{x',\alpha}$ have nonempty intersection,
        \[
            \abs{B_{x,\alpha}\cap B_{x',\alpha}}=\frac{\varepsilon\abs{\mc S}}{\abs{\mc A}}.
        \]
    \end{enumerate}
    The first two conditions imply that $\overline\lambda$ equals the $\lambda$ from the statement of the corollary. Together with the constant size of the blocks $B_{x,\alpha}$, which is guaranteed by property (ACFU1), condition 2) says that for each $\alpha$, the incidence relation $D_\alpha$ on the point set $\mc S$ with block set $\{B_{x,\alpha}:x\in\mc X\}$ is a BIBD$(v,k,\lambda)$ with parameters as claimed. The third of the above conditions for equality implies that any two distinct blocks $B_{x,\alpha}$ and $B_{x',\alpha}$ intersect in either $0$ or $\varepsilon\abs{\mc S}/\abs{\mc A}$ points. If the intersection size were equal to $\varepsilon\abs{\mc S}/\abs{\mc A}$ for all distinct blocks, then the $D_\alpha$ would be symmetric, in particular, $b=v$. This is impossible by Corollary \ref{cor:OCFU_S_lb} below. Hence, $D_\alpha$ must be quasi-symmetric.

    Now assume we are given a mosaic $\tilde M=(\tilde D_\alpha)_{\alpha\in\mc A}$ of quasi-symmetric BIBD$(v,k,\lambda)$ on the point set $\mc S$ and with block index set $\mc X$, with intersection numbers $0$ and $\mu>0$. It is well-known that 
    \begin{equation}\label{eq:QS_des_inters}
    	\mu=\frac{(k-1)(\lambda-1)}{r-1}+1
    \end{equation}
    (e.g., \cite[Proposition 3.17]{ShrSin_QS_Des}). Let $M$ be the dual of $\tilde M$ and $f_M$ the function induced by $M$. Then, using the notation \eqref{eq:dual_blocks} with $f=f_M$, it holds for any $\alpha\in\mc A$ and distinct points $x,x'\in\mc X$ that
    \[
    	\abs{\{s\in\mc S:f_M(x,s)=f_M(x',s)=\alpha\}}=\abs{B_{x,\alpha}\cap B_{x',\alpha}}\leq\mu.
    \]
    Defining $\varepsilon$ as in the statement then makes $f_M$ an $\varepsilon$-ACFU hash function. Inserting this in the right-hand side of \eqref{eq:seed_size_lb_variance} and solving for $\abs{\mc S}=v$ using \eqref{eq:QS_des_inters} shows that equality in \eqref{eq:seed_size_lb_variance} is satisfied. Obviously, $f_M$ is nontrivial since $\abs{\mc A}=v/k\geq 2$. 
\end{proof}

We have not been able to characterize those ACFU hash functions $f$ satisfying equality in \eqref{eq:seed_size_lb_simple}. Combinatorially, equality is equivalent to the statement that $\abs{B_{x,\alpha}\cap B_{x',\alpha}}\leq 1$ for all $x\neq x'$. This immediately implies $\abs{B_{s,\alpha}\cap B_{s',\alpha}}\leq 1$ for $s\neq s'$, so if the $B_{s,\alpha}$ have constant size, then the dual $\tilde M(f)$ of $M(f)$ also gives rise to an ACFU hash function satisfying equality in \eqref{eq:seed_size_lb_simple}. If the members $\tilde D_\alpha$ of $\tilde M(f)$ in addition are BIBDs (i.e., $\abs{B_{x,\alpha}\cap B_{x',\alpha}}=1$ for all distinct $x,x'$), then they are quasi-symmetric, and in this case $f$ satisfies both bounds of Theorem \ref{thm:seed_size_lb}, see Example \ref{ex:aff_OCFU}. So to find examples where only \eqref{eq:seed_size_lb_simple} is satisfied, other types of designs need to be considered. ACFU hash functions whose mosaics consist of transversal designs or nets are given in Example \ref{ex:TD}. 

When does each of the two bounds given in Theorem \ref{thm:seed_size_lb} apply? For the purpose of this discussion, given $\abs{\mc X}$ and $\abs{\mc A}$, let us call $\varepsilon$ \textit{feasible} if it satisfies the inequality of Lemma \ref{lem:epsilon_lb} and if $\varepsilon\leq 1$ (which we can assume without loss of generality). Let us also say that the inequality \eqref{eq:seed_size_lb_variance} \textit{applies} to $\varepsilon$ if the right-hand side of \eqref{eq:seed_size_lb_variance} is larger than that of \eqref{eq:seed_size_lb_simple}; otherwise, we say that \eqref{eq:seed_size_lb_simple} applies. 

\begin{lem}\label{lem:ACFU_lb_distinction}
	Let $\abs{\mc X}>\abs{\mc A}>1$. The bound \eqref{eq:seed_size_lb_variance} applies to those feasible $\varepsilon$ satisfying
    \[
    	\frac{\abs{\mc X}-\abs{\mc A}}{\abs{\mc A}(\abs{\mc X}-1)}
    	\leq\varepsilon
    	\leq\frac{\abs{\mc X}-\abs{\mc A}^2}{\abs{\mc X}-\abs{\mc A}}.
    \]
	This set is nonempty if and only if
	\begin{equation}\label{eq:lb_comparison_real}
        \abs{\mc X}\geq\frac{\abs{\mc A}}{2}\left(\abs{\mc A}+\sqrt{(\abs{\mc A}+3)(\abs{\mc A}-1)}+1\right).
    \end{equation}
\end{lem}

\begin{proof}
	The right-hand side of \eqref{eq:seed_size_lb_simple} is larger than that of  \eqref{eq:seed_size_lb_variance} if and only if $\varepsilon$ satisfies
    \[
        \frac{\abs{\mc X}-\abs{\mc A}^2}{\abs{\mc X}-\abs{\mc A}}\leq\varepsilon\leq 1.
    \]
    Comparing this with the optimal $\varepsilon$ from Lemma \ref{lem:epsilon_lb}, we obtain the criterion \eqref{eq:lb_comparison_real}.
\end{proof}

The case which interests us most is where $\varepsilon\approx 1/\abs{\mc A}$ (see Section \ref{sect:motivation}). Then if \eqref{eq:seed_size_lb_simple} applies to $\varepsilon$, the cardinality of $\mc S$ cannot be much smaller than $\abs{\mc A}^2$. But the right-hand side of \eqref{eq:lb_comparison_real} is approximately equal to $\abs{\mc A}^2$ for large $\abs{\mc A}$, so in this case, $\abs{\mc S}$ will usually be larger than $\abs{\mc X}$. On the other hand, the lower bound \eqref{eq:seed_size_lb_variance} is at most $\abs{\mc X}$, and Example \ref{ex:dual_aff} below shows that $\abs{\mc S}<\abs{\mc X}$ is indeed possible since it satisfies equality in \eqref{eq:seed_size_lb_variance} with $\varepsilon=1/\abs{\mc A}$. However, for optimal $\varepsilon$, the situation is special and the seed set is strictly larger than the point set, as we will see in the next subsection.

\subsection{Optimal $\varepsilon$}

We now consider the case where $\varepsilon$ is the optimal one from Lemma \ref{lem:epsilon_lb}, i.e., we deal with the case of OCFU hash functions. Stinson's result, Lemma \ref{lem:Stinson_lemma}, shows that an OCFU hash function $f$ has a corresponding $\Sigma M(f)$ which is a BIBD. The next result characterizes OCFU hash functions by the members of $M(f)$.

\begin{thm}
	If $f:\mc X\times\mc S\to\mc A$ is an OCFU hash function, then $M(f)$ is a mosaic of BIBDs $(v,k,\lambda)$ with
	\[
		v=\abs{\mc X},\quad k=\frac{\abs{\mc X}}{\abs{\mc A}},\quad\lambda=\frac{\varepsilon\abs{\mc S}}{\abs{\mc A}},
		\quad b=\abs{\mc S},\quad r=\frac{\abs{\mc S}}{\abs{\mc A}}.
	\]
	Conversely, if $M=(D_\alpha)_{\alpha\in\mc A}$ is a mosaic of BIBD$(v,k,\lambda)$ on $\mc X$ with block index set $\mc S$, then $f_M$ is an OCFU hash function, and $\abs{\mc X},\abs{\mc S},\abs{\mc A}$ satisfy the above relations.
\end{thm}

\begin{proof}
	Let $f:\mc X\times\mc S\to\mc A$ be an OCFU hash function. Since $\Sigma M(f)$ is a BIBD, all blocks $B_{s,\alpha}$ have fixed size $k$. Property (ACFU1) also ensures that the dual blocks $B_{x,\alpha}$ have constant size $\abs{\mc S}/\abs{\mc A}$, which we denote by $r$.
	
	We need to check that two points meet in precisely $\lambda$ blocks for suitable $\lambda$. To do this, fix $\alpha\in\mc A$ and $x\in\mc X$ and define the weights
	\[
		\kappa_i=\abs{\{x'\in\mc X:x'\neq x,\abs{B_{x,\alpha}\cap B_{x',\alpha}}=i\}}.
	\]
    Set 
    \[
        L
        =\lfloor\varepsilon\abs{\mc S}/\abs{\mc A}\rfloor
        =\left\lfloor\frac{r(k-1)}{\abs{\mc X}-1}\right\rfloor,
    \]
    the largest possible value of $i$. Note that
    \[
    	\sum_{i=0}^L\kappa_i=\abs{\mc X}-1.
    \]
    By counting pairs $(x',s)$ with $x'\neq x$ satisfying $s\in B_{x,\alpha}\cap B_{x',\alpha}$, we obtain
    \begin{align*}
    	\frac{1}{\abs{\mc X}-1}\sum_{i=0}^Li\kappa_i
    	&=\frac{r(k-1)}{\abs{\mc X}-1}.
    \end{align*}
    This implies that $\kappa_L=\abs{\mc X}-1$ and $\kappa_i=0$ for $i<L$. It follows that $\abs{B_{x',\alpha}\cap B_{x,\alpha}}=L$ for all $\alpha$ and $x\neq x'$. Hence $D_\alpha$ is a BIBD$(v,k,\lambda)$. The fact that 
    \[
    	\lambda=L=\frac{r(k-1)}{\abs{\mc X}-1}=\frac{\varepsilon\abs{\mc S}}{\abs{\mc A}}
    \]
    follows from \eqref{eq:lambda_v_r_k}.

    In the other direction, given a mosaic $M$ of BIBDs, it is straightforward to check that $f_M$ is an OCFU hash function and that the parameters are related as claimed in the statement.
\end{proof}

\begin{cor}\label{cor:OCFU_S_lb}
	For an OCFU hash function $f:\mc X\times\mc S\to\mc A$,
	\begin{equation}\label{eq:seed_size_lb_OCFU}
		\abs{\mc S}\geq\frac{\abs{\mc A}(\abs{\mc X}-1)}{\abs{\mc A}-1}.
	\end{equation}
	Equivalently, in a mosaic of BIBD$(v,k,\lambda)$, it holds that $b\geq v+r-1$.
\end{cor}

\begin{proof}
	In a mosaic $M$ of BIBD$(v,k,\lambda)$, the block size $k$ of each member $D_\alpha$ of $M$ divides the size $v$ of the point set. Since each $D_\alpha$ is a BIBD, the result proved independently by Roy \cite{Roy_resolvable} and Mikhail \cite{Mikhail_resolvable} applies, stating that $b\geq v+r-1$ in this case. With $a=v/k$, this is equivalent to $(a-1)b/a=(a-1)r\geq v-1$ (recall \eqref{eq:1-design}), which can be transformed into inequality \eqref{eq:seed_size_lb_OCFU} for $f_M$.
\end{proof}

The bound given in the corollary is tight, as will be seen in Example \ref{ex:aff_OCFU} below. However, it can only be attained in the regime where $\abs{\mc X}\geq\abs{\mc A}^2$, since the right-hand side of \eqref{eq:seed_size_lb_simple} is strictly larger than the right-hand side of \eqref{eq:seed_size_lb_OCFU} if $\abs{\mc X}<\abs{\mc A}^2$. For the latter case, an OCFU function is given in Example \ref{ex:Denniston}.

We also note that the corollary implies that the block index set of a nontrivial mosaic of BIBDs is strictly larger than its point set. Therefore the member BIBDs cannot be symmetric. This proves the claim that the designs achieving equality in \eqref{eq:seed_size_lb_variance} are truly quasi-symmetric, completing the proof of Corollary \ref{cor:eq_bed_variance}.

\subsection{Examples}

\begin{ex}\label{ex:aff_OCFU}
	It was shown in \cite{GGP_mosaics} that a mosaic of BIBDs $(D_\alpha)$ can be constructed from any resolvable BIBD $D$ in such a way that every $D_\alpha$ is isomorphic to $D$. Starting with an affine design, this gives a mosaic of designs satisfying equality in Corollary \ref{cor:OCFU_S_lb}. We recall the explicit form of a corresponding OCFU function, originally given in \cite{mosaics}. Let $q$ be a prime power and $\mbb F_q$ the field with $q$ elements. For any positive integer $t$, let $D$ be the affine designs on $\mbb F_q^t$ with blocks given by the hyperplanes (cosets of $(t-1)$-dimensional subspaces)  of $\mbb F_q^t$. Every $(t-1)$-dimensional subspace of $\mbb F_q^t$ can be identified with the solution space of the equation $\sum_ih_ix_i=0$, where $h$ is a unique nonzero vector in $\mbb F_q^t$ whose first nonzero component is 1. Denote the set of such vectors by $\mc H$ and define $f:\mbb F_q^t\times(\mc H\times\mbb F_q)\to\mbb F_q$ by
	\[
		f(x;h,\beta)=\sum_ih_ix_i+\beta. 
	\]
    For fixed $\alpha\in\mbb F_q$, as the pair $(h,\beta)$ ranges over all possible values, the preimages $f(\cdot;h,\beta)^{-1}(\alpha)=B_{h,\beta;\alpha}$ range over all hyperplanes of $\mbb F_q^t$. Hence the incidence structure $D_\alpha$ on $\mc X$ formed by the block set $\{B_{h,\beta;\alpha}:h\in\mc H,\beta\in\mbb F_q\}$ and the $\in$ relation is isomorphic to $D$, and the family $M=(D_\alpha)_{\alpha\in\mbb F_q}$ is a mosaic of BIBDs. Thus $f$ is an OCFU hash function.
    
    The example shows that the bound from Corollary \ref{cor:OCFU_S_lb} is tight (not surprisingly, given \eqref{eq:aff_des_blocks}). In case $t=2$, which corresponds to a mosaic of affine planes, equality holds in \eqref{eq:seed_size_lb_simple} as well. 
\end{ex}

\begin{ex}\label{ex:dual_aff}
    Consider the dual $\tilde M$ of the mosaic $M$ of affine BIBDs considered in the previous example. Since two distinct non-parallel hyperplanes meet in $q^{t-2}$ points, $f=f_{\tilde M}$ is a $1/\abs{\mc A}$-ACFU hash function. Since an affine design is quasi-symmetric, $f$ attains the bound \eqref{eq:seed_size_lb_variance}. If $f$ derives from the dual of a mosaic of affine planes (i.e., $t=2$), then it also satisfies equality in \eqref{eq:seed_size_lb_simple}. Written as a function with the same notation as in the previous example, $f$ satisfies the formula
    \[
        f:(\mc H\times\mbb F_q)\times\mbb F_q^t\to\mbb F_q,\quad f(h,\beta;x)=\sum_{i=1}^th_ix_i+\beta.
    \]
\end{ex}

\begin{ex}\label{ex:Denniston}
    Any OCFU hash function $f:\mc X\times\mc S\to\mc A$ whose underlying mosaic $M(f)$ is a mosaic of BIBD$(v,k,1)$ (with $v=\abs{\mc X},k=\abs{X}/\abs{\mc A}$) satisfies equality in \eqref{eq:seed_size_lb_simple}. In this case, $\abs{\mc A}^2$ has to be at least as large as $\abs{\mc X}$ by the discussion after Corollary \ref{cor:OCFU_S_lb}. An example was studied in \cite[p.~608]{mosaics}, based on the resolvable designs arising from Denniston's construction of maximal arcs in projective space over $\mbb F_2$. Unfortunately, the corresponding OCFU hash function does not have a nice closed form, but it was shown in \cite{mosaics} to be polynomial-time computable. The size of $\abs{\mc A}$ ranges between $\sqrt{\abs{\mc X}}$ and $\abs{\mc X}$.
\end{ex}

\begin{ex}\label{ex:TD}
    We would also like to have $1/\abs{\mc A}$-ACFU hash functions with a small seed set in the range where \eqref{eq:seed_size_lb_simple} is the relevant lower bound for the seed size, i.e., where $\abs{\mc A}(\abs{\mc A}+1)\geq\abs{\mc X}$. Again let $q$ be a prime power, let $\mc H$ be a subset of $\mbb F_q$, set $\mc X=\mc H\times\mbb F_q$ as well as $\mc S=\mbb F_q^2$ and $\mc A=\mbb F_q$. Then define \cite[p.~610]{mosaics}
    \[
        f(h,y;s_1,s_2)=s_2-hs_1+y.
    \]
    Since for distinct $(h,y)$ and $(h',y')$ and any $\alpha\in\mbb F_q$ there exists at most one $(s_1,s_2)$ such that $f(h,y;s_1,s_2)=f(h',y';s_1,s_2)=\alpha$, this function satisfies equality in \eqref{eq:seed_size_lb_simple}. (If we set $\mc H=\mbb F_q$, additionally extend it by the symbol $\infty$ and define $f(\infty,y,s_1,s_2)=s_1+y$, then we obtain the case $t=2$ from Example \ref{ex:dual_aff}.)
    
    The combinatorial structure of $f$ is as follows. We can partition $\mc X$ into the $\abs{\mc H}$ point classes $\{(h,y):y\in\mbb F_q\}$. If two distinct points $(h,y)$ and $(h',y')$ are from the same point class, then no block of $M(f)=(D_\alpha)_{\alpha\in\mc A}$ contains both of them. If they come from different point classes, then for each $\alpha\in\mc A$, there exists a unique block from $D_\alpha$ containing them both. That means that each $D_\alpha$ is isomorphic to the same transversal design, and the members of $M(f)$ even share the same point class partition. 
    
    By passing to the dual $\tilde f:\mc S\times\mc X\to\mc A$, we obtain a seed-optimal $1/\abs{\mc H}$-ACFU hash function. $M(\tilde f)=\tilde M(f)$ is a mosaic of nets.
\end{ex}

\subsection{Comparison with AU and ASU bounds}

For completeness and comparison, we briefly discuss lower bounds on the size of the seed set of an AU or ASU function. 

\begin{lem}[\cite{Stinson_UH_AC}, Theorems 4.1 and 4.3]\label{lem:seed_size_AU_ASU}
    Let $\varepsilon>0$. 
    \begin{enumerate}
        \item For any $\varepsilon$-AU hash function $f:\mc X\times\mc S\to\mc A$,
        \begin{equation}\label{eq:AU_seed_size}
            \abs{\mc S}
            \geq\frac{\abs{\mc X}(\abs{\mc A}-1)}{\varepsilon\abs{\mc A}(\abs{\mc X}-\abs{\mc A})+\abs{\mc A}^2-\abs{\mc X}}.
        \end{equation}
        \item For any $\varepsilon$-ASU hash function $f:\mc X\times\mc S\to\mc A$,
        \begin{equation}\label{eq:ASU_seed_eq}
            \abs{\mc S}
            \geq1+\frac{\abs{\mc X}(\abs{\mc A}-1)^2}{\varepsilon\abs{\mc A}(\abs{\mc X}-1)+\abs{\mc A}-\abs{\mc X}}.
        \end{equation}
    \end{enumerate}
\end{lem} 

\begin{table}[!t]
\footnotesize
\centering
\captionsetup{width=.8\linewidth}
{
\begin{tabular}{M{2cm}|M{1.5cm}|M{4.5cm}|M{3cm} N}
	$\varepsilon$ & AU & ACFU & ASU & \\ [0.1cm]
	\hline 
	\vspace{1em} optimal & \vspace{0.5em} $\dfrac{\abs{\mc X}-1}{\abs{\mc A}-1}$ & \vspace{0.5em} $\dfrac{\abs{\mc A}(\abs{\mc X}-1)}{\abs{\mc A}-1}$ & \vspace{1em} --- & \\[0.5cm]
	\hline
	\vspace{0.5em} $\dfrac{1}{\abs{\mc A}}$ & \vspace{0.5em} $\dfrac{\abs{\mc X}}{\abs{\mc A}}$ & \vspace{0.5em} $\max\left\{\abs{\mc A}^2,1+\dfrac{\abs{\mc X}(\abs{\mc A}-1)}{\abs{\mc A}}\right\}$ & \vspace{1em} $1+\abs{\mc X}(\abs{\mc A}-1)$ & \\[0.2cm]
\end{tabular}}
\caption{The lower bounds for optimal $\varepsilon$ and $\varepsilon=1/\abs{\mc A}$. For ASU hash functions, $\varepsilon\geq1/\abs{\mc A}$ \cite{Stinson_UH_AC}.}
\label{table}
\end{table}

We also have an additional simple lower bound for the seed size of ASU hash functions analogous to \eqref{eq:seed_size_lb_simple}, which to our knowledge has not yet been stated explicitly anywhere.

\begin{lem}\label{lem:ASU_seed_simple}
	Let $f:\mc X\times\mc S\to\mc A$ be an $\varepsilon$-ASU hash function. Then
        \[
        	\abs{\mc S}\geq \frac{\abs{\mc A}}{\varepsilon}.
        \]
\end{lem}

\begin{proof}
    Assume $f(x,s)=\alpha$. For any distinct $x'\in\mc X$, there must exist $\alpha'\in\mc A$ such that $f(x',s)=\alpha'$. Therefore $\varepsilon\abs{\mc S}/\abs{\mc A}\geq 1$.
\end{proof}

It is simple to check that the relation between the ASU bounds is as follows.

\begin{lem}\label{lem:ASU_seed_distinction}
	For an $\varepsilon$-ASU hash function $f:\mc X\times\mc S\to\mc A$ with $\abs{\mc X}>\abs{\mc A}$, the right-hand side of the bound from Lemma \ref{lem:ASU_seed_simple} is larger than the one from \eqref{eq:ASU_seed_eq} if and only if
	\[
		\frac{\abs{\mc X}-\abs{\mc A}}{\abs{\mc X}-1}\leq\varepsilon\leq 1.
	\]
\end{lem}

Note that, if $\abs{\mc X}>\abs{\mc A}$, the left-hand side of the inequality of Lemma \ref{lem:ASU_seed_distinction} is always at least $1/\abs{\mc A}$. Hence given $\abs{\mc X}$ and $\abs{\mc A}$, there is always a range of $\varepsilon$ sufficiently close to $1/\abs{\mc A}$ where the bound from Lemma \ref{lem:seed_size_AU_ASU} applies.

\begin{rem}\label{rem:AU_simple}
	The analogous lower bound $\abs{\mc S}\geq1/\varepsilon$ for $\varepsilon$-AU hash functions gives no new information, since $1/\varepsilon$ is always smaller than the right-hand side of \eqref{eq:AU_seed_size}, with equality if and only if $\abs{\mc X}=\abs{\mc A}^2$.
\end{rem}

We are again interested in conditions for equality in the above bounds. For the OU case, Stinson gives the following criterion.

\begin{lem}[\cite{Stinson_comb_tech_UH}, Theorem 2.2]\label{lem:AU_eq_bed}
	If $f$ is an OU hash function satisfying equality in \eqref{eq:AU_seed_size}, then $\Sigma M(f)$ is an affine BIBD. Conversely, if $D$ is any affine BIBD, then the function $f_M$ induced by the mosaic $M$ corresponding to $D$ by Lemma \ref{lem:mos_res} is an OU hash function which satisfies equality in \eqref{eq:AU_seed_size}.
\end{lem}

No result is known to us which characterizes $\varepsilon$-AU hash functions satisfying equality in \eqref{eq:AU_seed_size} for general $\varepsilon$.

For the ASU case of Lemma \ref{lem:seed_size_AU_ASU}, van Trung determines the condition for equality. Our proof of Corollary \ref{cor:eq_bed_variance} is similar to the proof of van Trung's result.

\begin{lem}[\cite{vanTrung_ASU}, Theorem 3.1]\label{lem:van_Trung_lem}
	An $\varepsilon$-ASU hash function $f:\mc X\times\mc S\to\mc A$ satisfies equality in \eqref{eq:ASU_seed_eq} if $\Sigma\tilde M(f)$ is a resolvable quasi-symmetric design with one intersection number equal to zero, where $\tilde M(f)$ is the dual of $M(f)$. In the other direction, if a resolvable quasi-symmetric design $D$ is given with one intersection number equal to $0$, then there exists an $\varepsilon$ such that $f_{\tilde M}$ is an $\varepsilon$-ASU hash function satisfying equality in \eqref{eq:ASU_seed_eq}, where $M$ is the mosaic determined by $D$ via Lemma \ref{lem:mos_res} and $\tilde M$ is its dual.
\end{lem}

The similarity of this result to ours on ACFU hash functions, Corollary \ref{cor:eq_bed_variance}, is no coincidence, as we will see in the next section (Theorem \ref{thm:block_set_ext}).

For Lemma \ref{lem:ASU_seed_simple}, we cannot characterize those ASU hash functions satisfying equality, but again there is a striking similarity with the ACFU situation (see the discussion after Corollary \ref{cor:eq_bed_variance}) which will be explained in the next section (Theorem \ref{thm:block_set_ext} again). However, a class of designs from which one can construct such ASU hash functions are the nets. A net $\tilde D$ on the point set $\mc S$ obviously is a resolvable incidence structure, so by Lemma \ref{lem:mos_res} it gives rise to a mosaic $\tilde M=(\tilde D_\alpha)_{\alpha\in\mc A}$. The blocks all have the same size and different blocks intersect in at most one point, so equality is satisfied in Lemma \ref{lem:ASU_seed_simple}. Let $\mc X$ be the block index set. Then the dual $M$ of $\tilde M$ induces a function $f_M$ which is an $\varepsilon$-ASU hash function for $\varepsilon=\abs{\mc A}/\abs{\mc S}$. If the net is an affine plane, then $f_M$ also satisfies equality in van Trung's bound. 

Table \ref{table} shows the bounds for the cases of optimal $\varepsilon$ as well as for $\varepsilon=1/\abs{\mc A}$. (Due to the application of ACFU hash functions we have in mind (see Section \ref{sect:motivation}), $\varepsilon\approx 1/\abs{\mc A}$ is what interests us most.) Interestingly, we need to consider the case of OCFU hash functions separately, whereas Stinson's lower bound for $\varepsilon$-AU hash functions also covers OU hash functions.

\section{Constructions of $\varepsilon$-ACFU hash functions}\label{sect:constructions}

\subsection{Extensions of AU and ASU hash functions}\label{sect:affext_examples}

While $\varepsilon$-ACFU hash functions are new, $\varepsilon$-AU and $\varepsilon$-ASU hash functions are well-investigated concepts with many efficiently computable examples, so it would be attractive to be able to turn an $\varepsilon$-AU or $\varepsilon$-ASU hash function into an $\varepsilon$-ACFU hash function. It turns out that this is indeed possible. 

Let $g:\mc X\times\mc H\to\mc A$ be an $\varepsilon$-AU hash function. Let $L$ be a latin square with entries from $\mc A$ and rows and columns indexed by $\mc A$, too. This can equivalently be described as a quasigroup structure on $\mc A$ whose product $\circ$ is defined by the rule $\alpha\circ\beta=L(\alpha,\beta)$. We denote the unique solution $\gamma$ of the equation $\gamma\circ\beta=\alpha$ by $\alpha/\beta$. We can now define the function $\hat g:\mc X\times(\mc H\times\mc A)\to\mc A$ by
\[
    \hat g(x;h,\beta)=g(x,h)\circ\beta,
\]
so its seed set is $\mc S=\mc H\times\mc A$. We call $\hat g$ the \textit{seed extension} of $g$.

\begin{thm}\label{thm:aff_ext}
    The function $g:\mc X\times\mc H\to\mc A$ is an $\varepsilon$-AU hash function if and only if its seed extension $\hat g$ is an $\varepsilon$-ACFU hash function. Each member of $M(\hat g)$ is isomorphic to $\Sigma M(g)$.
\end{thm}

\begin{proof}
    The equation $\hat g(x;h,\beta)=\alpha$ means that $g(x,h)=\alpha/\beta$. Moreover, as $\beta$ varies over $\mc A$, the unique solution $\alpha/\beta$ of $\gamma\circ\beta=\alpha$ assumes all possible values in $\mc A$. Thus for any $x,x'\in\mc X$,
    \begin{align}
        \abs{\{(h,\beta):\hat g(x;h,\beta)=\hat g(x';h,\beta)=\alpha\}}
        =\abs{\{h:g(x,h)=g(x',h)\}}.\label{eq:reg_gen}
    \end{align}
    If $x=x'$, then the set in \eqref{eq:reg_gen} is all of $\mc H$, whose size is $\abs{\mc S}/\abs{\mc A}$. Thus $\hat g$ always satisfies property (ACFU1). Now assume that $x\neq x'$. One sees immediately from \eqref{eq:reg_gen} that $g$ is an $\varepsilon$-AU hash function if and only if $\hat g$ is an $\varepsilon$-ACFU hash function.
    
    Finally, consider the $\alpha$-th member $D_\alpha$ of $M(\hat g)$. Its blocks have the form
    \[
    	B_{h,\beta;\alpha}
    	=\{x:\hat g(x;h,\beta)=\alpha\}
    	=\{x:g(x,h)=\alpha/\beta\},
    \]
    for $h\in\mc H$ and $\beta\in\mc A$. As $h,\beta$ vary over all possible values, the blocks $B_{h,\beta;\alpha}$ vary over all blocks of $\Sigma M(g)$. This shows that $D_\alpha$ is isomorphic to $\Sigma M(g)$.
\end{proof}

As a corollary, we obtain the result of Gnilke, Greferath and Pav\v cevi\'c on the relation between resolvable BIBDs and mosaics of BIBDs. 

\begin{cor}[\cite{GGP_mosaics}, Theorem 3.4]\label{cor:mos_constr_primal}
	For any resolvable BIBD$(v,k,\lambda)$ $D$, there exists a mosaic of BIBD$(v,k,\lambda)$ each of whose members is isomorphic to $D$.
\end{cor}

The ACFU hash functions from Examples \ref{ex:aff_OCFU}, \ref{ex:Denniston} and \ref{ex:TD} can be constructed from AU hash functions in this way. Note that if the OU hash function $g$ satisfies equality in \eqref{eq:AU_seed_size}, then $\hat g$ necessarily satisfies equality in \eqref{eq:seed_size_lb_OCFU}. If $\varepsilon$ is strictly larger than the optimal one, an $\varepsilon$-AU hash function can almost never produce a seed-optimal $\hat g$. Equality in \eqref{eq:seed_size_lb_variance} would only be possible if the $\abs{\mc A}$-fold multiple of the right-hand side of \eqref{eq:AU_seed_size} were smaller than the right-hand side of \eqref{eq:seed_size_lb_variance}, but this can only hold for the trivial case of $\varepsilon=1$. Equality in \eqref{eq:seed_size_lb_simple}, i.e., the relation $\abs{\mc S}=\abs{\mc A}/\varepsilon$, would require $\abs{\mc H}=1/\varepsilon$. By Remark \ref{rem:AU_simple}, this requires that \eqref{eq:AU_seed_size} is satisfied as well and that $\abs{\mc X}=\abs{\mc A}^2$. By Lemma \ref{lem:AU_eq_bed}, the corresponding $\Sigma M(f)$ must be an affine plane.

\begin{rem}\label{rem:extension}
    Assume $\mc X$ and $\mc A$ are groups and that the function $g:\mc X\times\mc H\to\mc A$ is a group homomorphism in the first argument for every fixed $h\in\mc H$. In addition, assume that 
    \begin{equation}\label{eq:AUtoASU_Groups}
    	\abs{\{h:g(x,h)=\alpha\}}\leq\varepsilon\abs{\mc H}
    \end{equation}
    for every $\alpha\in\mc A$ and every $x\in\mc X\setminus\{e_{\mc X}\}$, where $e_{\mc X}$ is the neutral element of $\mc X$. Then $g$ is an $\varepsilon$-AU hash function, since if $\alpha$ is the neutral element of $\mc A$, the left-hand side of \eqref{eq:AUtoASU_Groups} is the same as $\abs{\{h:g(x_1,h)=g(x_2,h)\}}$ for any $x_1,x_2$ such that $x_1x_2^{-1}=x$. Moreover, $\hat g$ is an $\varepsilon$-ASU hash function, where $\hat g$ is constructed with respect to the given group structure on $\mc A$, which can be seen by proceeding analogously to the proof of Theorem \ref{thm:aff_ext} and using \eqref{eq:AUtoASU_Groups}. 
    
    This statement was proved by Krawczyk in \cite{Krawczyk_LFSR_UH}; he calls $g$ \textit{$\varepsilon$-balanced} if it satisfies \eqref{eq:AUtoASU_Groups}. This construction of ASU hash functions has been applied frequently in the literature. For instance, Stinson constructs a $1/\abs{\mc A}$-ASU hash function in this way in the proof of \cite[Theorem 5.2]{Stinson_UH_AC}. Two specific examples are given in Examples \ref{ex:Toeplitz} and \ref{ex:Hay_fct}.
\end{rem}

\begin{ex}\label{ex:Toeplitz}
    In an $m\times n$ Toeplitz matrix $T=(t_{ij})$ over the field $\mbb F_q$ of size $q$, the entry $t_{ij}$ only depends on $i-j$. Thus $T=T_h$ is determined by the vector $h$ of the $m+n-1$ entries in the first row and first column. Define
    \[
        g:\mbb F_q^n\times F_q^{n+m-1}\to\mbb F_q^m,
        \quad g(x,h)=T_hx.
    \]
    It was shown in \cite[Claim 2.2]{MNT_complexity_UH} that $g$ satisfies \eqref{eq:AUtoASU_Groups} with $\varepsilon=1/\abs{\mc A}$ and that $\hat g$ is  an $\varepsilon$-ASU hash function. In terms of seed length, both $g$ and $\hat g$ are suboptimal.
\end{ex}

\begin{ex}\label{ex:Hay_fct}
    For any prime power $q$, define the function 
    \[
        g:\mbb F_{q^n}\times\mbb F_{q^n}\to\mbb F_q^m,
        \quad g(x,h)=(hx)_m,
    \]
    where $(hx)_m$ is the vector consisting of the first $m$ components of a representation of $hx$ as an $n$-dimensional vector over $\mbb F_q$. This $g$ satisfies \eqref{eq:AUtoASU_Groups} for $\varepsilon=q^{-m}=1/\abs{\mc A}$ and is linear in $x$ with $x$ regarded as an element of $\mbb F_q^n$ for every $h$, hence $\hat g$ is a $1/\abs{\mc A}$-ASU hash function. For $q$ prime, $\hat g$ was already defined in \cite{CW_UH} and recognized as a $1/\abs{\mc A}$-ASU hash function in \cite{WC_UH}. Stinson also uses it in \cite[Theorem 5.2]{Stinson_UH_AC} to construct an ASU hash function.
    
    By excluding $h=0$, one obtains a function $g_*$ for which all preimages $\{x:g_*(h,x)=\alpha\}$ have the same size. In fact, this modification makes $g_*$ an OU hash function \cite[Appendix B]{BT_poly_time}. Hence, the corresponding $\hat g_*$ is another example of an OCFU hash function, although it satisfies neither \eqref{eq:seed_size_lb_OCFU} nor \eqref{eq:seed_size_lb_simple}. $\hat g_*$ also is an $\varepsilon$-ASU hash function, but with $\varepsilon=q^{n-m}/(q^n-1)>q^{-m}=1/\abs{\mc A}$.  
    
    The function $g_*$ itself was used in \cite{BTV_published,BT_poly_time} to show semantic security for symmetric wiretap channels, an application related to the one we present in Section \ref{sect:motivation}. A variant of $\hat g_*$, but with a larger seed, was used in \cite[Remark 16, Lemma 21]{HayMat} in the context of information-theoretic security. From the practical viewpoint, $g$ and $g_*$ as well as their extended versions $\hat g,\hat g_*$ have the advantage that $\abs{\mc A}=q^m$ can be any power of $q$ between $1$ and $\abs{\mc X}=q^n$. 
\end{ex}

It is not only possible to turn an AU hash function into an ACFU hash function by extending its seed set. Let $g:\mc Y\times\mc S\to\mc A$ be an $\varepsilon$-ASU hash function  and equip $\mc A$ with a quasigroup product denoted by $\circ$. Setting $\mc X=\mc Y\times\mc A$, we define the \textit{point extension} of $g$ by $\check g:\mc X\times\mc S\to\mc A$,
\[
	\check g(y,\beta;s)=g(y,s)\circ\beta.
\]
To connect this with the seed extension of functions, for any function $g:\mc Y\times\mc S\to\mc A$, let $\tilde g:\mc S\times\mc Y\to\mc A$ be defined by $\tilde g(s,y)=g(y,s)$ (the tilde makes sense since $M(\tilde g)=\tilde M(g)$, the dual of the mosaic of $g$.) Then we have the relation 
\[
	\tilde{\check g}=\hat{\tilde g}.
\]
The point extension of functions $g$ where $\Sigma\tilde M(g)$ is a resolvable BIBD or GDD was already considered in \cite{mosaics}.

\begin{thm}\label{thm:block_set_ext}
	A function $g:\mc Y\times\mc S\to\mc A$ is an $\varepsilon$-ASU hash function if and only if its point extension $\check g$ is an $\varepsilon$-ACFU hash function. Each member of $M(\check g)$ is isomorphic to the dual of $\Sigma\tilde M(g)$, where $\tilde M(g)$ is the dual of $M(g)$.
\end{thm}

\begin{proof}
    Assume that $g$ is $\varepsilon$-ASU. By the definition of $\check g$,
	\[
		\abs{\{s:\check g(y,\beta;s)=\check g(y',\beta';s)=\alpha\}}
		=\abs{\{s:g(y,s)=\alpha/\beta,g(y',s)=\alpha/\beta'\}}.
	\]
    This expression equals $\abs{\mc S}/\abs{\mc A}$ if $(y,\beta)=(y',\beta')$ by (ASU1). If $y=y'$, but $\beta\neq\beta'$, then it equals zero; if $y\neq y'$, it is upper-bounded by $\varepsilon\abs{\mc S}/\abs{\mc A}$ by (ASU2). Hence, $\check g$ is an $\varepsilon$-ACFU function.
    
    Now assume that $\check g$ is $\varepsilon$-ACFU. For any $y,y'\in\mc Y$ and $\alpha,\alpha'\in\mc A$, consider the set
    \[
    	\{s:g(y,s)=\alpha,g(y',s)=\alpha'\}.
    \]
    For any pair $\beta,\beta'\in\mc A$, this is the same as
    \begin{equation}\label{eq:ASU_generic_set}
    	\{s:g(y,s)\circ\beta=\alpha\circ\beta,g(y',s)\circ\beta'=\alpha'\circ\beta'\}.
    \end{equation}
    If $y=y'$ as well as $\alpha=\alpha'$ and $\beta=\beta'$, this set has the same cardinality as the set 
    \[
    	\{s:\check g(y,\beta;s)=\alpha\circ\beta\}.
    \]
    Thus the property (ACFU1) of $\check g$ gives us (ASU1) for $g$. 
    
    On the other hand, if $y\neq y'$, we can choose $\beta$ and $\beta'$ in such a way that $\alpha\circ\beta=\alpha'\circ\beta'$, so that the set \eqref{eq:ASU_generic_set} has the same cardinality as
    \[
    	\{s:\check g(y,\beta;s)=\check g(y',\beta';s)=\alpha\circ\beta\},
    \]
    which by (ACFU2) is upper-bounded by $\varepsilon\abs{\mc S}/\abs{\mc A}$. This implies (ASU2).    
    
    Finally, we prove the statement about the members of $M(\check g)$. To see this, note that $\tilde M(\check g)=M(\tilde{\check g})=M(\hat{\tilde g})$. By Theorem \ref{thm:aff_ext}, the members of this mosaic are isomorphic to $\Sigma M(\tilde g)=\Sigma\tilde M(g)$.   
\end{proof}

If we are given a resolvable quasi-symmetric BIBD, then we know from Corollary \ref{cor:mos_constr_primal} that one can construct a mosaic of quasi-symmetric BIBDs from this. Taking duals, one gets from a seed-optimal $\varepsilon$-ASU hash function (by Lemma \ref{lem:van_Trung_lem}) to a seed-optimal $\varepsilon$-ACFU hash function (by Corollary \ref{cor:eq_bed_variance}). Theorem \ref{thm:block_set_ext} does this in a single step, and (necessarily) the lower bound \eqref{eq:ASU_seed_eq} transforms into \eqref{eq:seed_size_lb_variance} in the right way. It is also obvious that an $\varepsilon$-ASU $g$ satisfies equality in Lemma \ref{lem:ASU_seed_simple} if and only if $\check g$ satisfies equality in \eqref{eq:seed_size_lb_simple}.

Examples \ref{ex:dual_aff} and \ref{ex:TD} show ACFU hash functions which can be constructed as point extensions of ASU hash functions. The function from Example \ref{ex:TD} can be represented both as the seed extension of an AU hash function and as the point extension of an ASU hash function. In fact, it is the typical example of such a function in the case where $\mc A$ is an abelian group. To see this, we extend Krawczyk's notion of $\varepsilon$-balancedness (see Remark \ref{rem:extension}) to arbitrary functions whose image lies in an abelian group. We say that a function $a:\mc Y\times\mc H\to\mc A$, where $\mc A$ is an abelian group, is $\varepsilon$\textit{-balanced} if for any two distinct $y,y'\in\mc Y$ and any $\beta\in\mc A$, it satisfies
\[
	\abs{\{h\in\mc H:a(y,h)-a(y',h)=\beta\}}\leq\varepsilon\abs{\mc H}.
\]
We also remark that if $f=\hat g_1=\check g_2$ and if it maps into $\mc A$, then it must have the form $f:(\mc Y\times\mc A)\times(\mc H\times\mc A)\to\mc A$.

\begin{prop}
	Let $f$ be an $\varepsilon$-ACFU hash function. Then $f$ can be represented as the seed extension of an $\varepsilon$-AU hash function $g_1:(\mc Y\times\mc A)\times\mc H\to\mc A$ and as the point extension of an $\varepsilon$-ASU hash function $g_2:\mc Y\times(\mc H\times\mc A)\to\mc A$ if and only if there exists an $\varepsilon$-balanced function $a:\mc Y\times\mc H\to\mc A$ such that
	\[
		f(y,\beta;h,\gamma)=a(y,h)+\beta+\gamma.
	\]
\end{prop}

\begin{proof}
	If $f=\hat g_1=\check g_2$, then for all $y,h,\beta,\gamma$,
	\[
		g_1(y,\beta;h)=g_2(y;h,\gamma)+\beta-\gamma.
	\]
    Since the left-hand side does not depend on $\gamma$, there is an $a(y,h)\in\mc A$ such that $g_1(y,\beta;h)=a(y,h)+\beta$, which implies that $f$ has the claimed form.
    
    We check that the function $a$ is $\varepsilon$-balanced. Let $y\neq y'$. Since $g_1$ is an $\varepsilon$-AU hash function, for any $\beta,\beta'$,
    \[
    	\abs{\{h:a(y,h)+\beta=a(y',h)+\beta'\}}\leq\varepsilon\abs{\mc H}.
    \]
    This inequality is precisely the definition of $\varepsilon$-balancedness since $\beta'-\beta$ can assume any value in $\mc A$.
    
    It is straightforward to check the converse. 
\end{proof}

Clearly, in Example \ref{ex:TD}, the function $a$ is multiplication of elements of $\mc R$ and of $\mbb F_q$.

\subsection{Concatenation}

We can show a result similar to the results on the concatenation of almost (strongly) universal hash functions in \cite{Stinson_UH_AC}.

\begin{prop}
    If $f_1:\mc X_1\times\mc S_1\to\mc A_1$ is an $\varepsilon_1$-ASU hash function and $f_2:\mc A_1\times\mc S_2\to\mc A_2$ is an $\varepsilon_2$-ACFU hash function, then the function $f:\mc X_1\times(\mc S_1\times\mc S_2)\to\mc A_2$ defined by
    \[
        f(x_1;s_1,s_2)=f_2(f_1(x_1,s_1),s_2)
    \]
    is an $(\varepsilon_1\varepsilon_2(\abs{\mc A_1}-1)+\varepsilon_1)$-ACFU hash function.
\end{prop}

\begin{proof}
    Let $x_1\in\mc X_1$ and $\alpha_2\in\mc A_2$. By (ACFU1), for each $\alpha_1\in\mc A_1$, the number of $s_2\in\mc S_2$ for which $f_2(\alpha_1,s_2)=\alpha_1$ equals $\abs{\mc S_2}/\abs{\mc A_2}$. The number of $s_1\in\mc S_1$ for which $f_1(x_1,s_1)=\alpha_1$ equals $\abs{\mc S_1}/\abs{\mc A_1}$ by (ASU1). Hence
    \begin{align*}
        \abs{\{(s_1,s_2):f(x_1;s_1,s_2)=\alpha_2\}}
        &=\frac{\abs{\mc S_1}}{\abs{\mc A_1}}\cdot\frac{\abs{\mc S_2}}{\abs{\mc A_2}}\cdot\abs{\mc A_1}
        =\frac{\abs{\mc S_1}\abs{\mc S_2}}{\abs{\mc A_2}},
    \end{align*}
    proving that $f$ satisfies (ACFU1). 
    
    To check (ACFU2), let $\alpha_2\in\mc A_2$ and choose distinct $x_1,x_1'\in\mc X_1$. We need to consider two cases. Suppose first that $f_1(x_1,s_1)\neq f_1(x_1',s_1)$. Then, $f_2(\cdot,s_2)$ needs to map these two distinct values to $\alpha_2$, and this is possible for at most $\varepsilon_2\abs{\mc S_2}/\abs{\mc A_2}$ values of $s_2$ by (ACFU2). If $f_1(s_1,x_1)=f_1(x_1',s_1)$, then $\abs{\mc S_2}/\abs{\mc A_2}$ values of $s_2$ are possible. Since $f_1$ is an ASU hash function, for any pair $\alpha_1,\alpha_1'$, there are at most $\varepsilon_1\abs{\mc S_1}/\abs{\mc A_1}$ possible $s_1$ such that  $f_1(x_1,s_1)=\alpha_1$ and $f_1(x_1',s_1)=\alpha_1'$. Therefore
    \begin{align*}
        &\abs{\{(s_1,s_2):f(x_1;s_1,s_2)= f(x_1';s_1,s_2)=\alpha_2\}}\\
        &\leq\frac{\varepsilon_1\abs{\mc S_1}}{\abs{\mc A_1}}\cdot\frac{\varepsilon_2\abs{\mc S_2}}{\abs{\mc A_2}}\cdot\abs{\mc A_1}(\abs{\mc A_1}-1)
        +\frac{\varepsilon_1\abs{\mc S_1}}{\abs{\mc A_1}}\cdot\frac{\abs{\mc S_2}}{\abs{\mc A_2}}\cdot\abs{\mc A_1}\\
        &=\bigl(\varepsilon_1\varepsilon_2(\abs{\mc A_1}-1)+\varepsilon_1\bigr)\frac{\abs{\mc S_1}\abs{\mc S_2}}{\abs{\mc A_2}}.
    \end{align*}
    This completes the proof.
\end{proof}

If we take $f_1$ to be an $\varepsilon$-ASU hash function with minimal $\varepsilon$, i.e., $\varepsilon=1/\abs{\mc A_1}$, and $f_2$ to be an OCFU hash function, then we obtain an $1/\abs{\mc A}$-ACFU hash function. However, the resulting seed set will always be larger than any of the bounds from Theorem \ref{thm:seed_size_lb}. We have not checked the situation for other $\varepsilon$, but we think it unlikely that this method of concatenation yields ACFU hash functions with a minimal seed.

The lemma gives the possibility of constructing new ACFU hash functions, e.g., with larger $\varepsilon$ than in most of the examples we have seen so far. (Although we do not have as yet any application for such functions.) For instance, Stinson \cite{Stinson_UH_AC} gives ASU hash functions with larger $\varepsilon$ which can be used as the first function in concatenation.

\section{Motivation from privacy amplification}\label{sect:motivation}

In this section, we sketch how ACFU hash functions can be used in privacy amplification with the goal of establishing a uniformly distributed key between two parties whose values are indistinguishable to an adversary.

The concept of privacy amplification goes back to Bennett, Brassard and Robert \cite{BBR_PA} and Bennett, Brassard, Cr\'epeau and Maurer \cite{BBCM_gen_PA}. In addition to the classical setup described in the introduction and below, it also plays an analogous role in quantum key distribution \cite{Assche_QKD}. Quite a lot of research on privacy amplification has been made in classical information theory, and various suggestions how to measure the security of the secret key have been considered \cite{MauW_extractors, ChouBloch_separability, Hay_almost_dual_UHF, Renes_PA_duality_channelcoding}. We are going to use a very strict security measure, which appears first in a special setting related to the quantum BB84 protocol \cite{Hay_BB84} and was not considered afterwards until recently \cite{mosaics}, and for which $\varepsilon$-ACFU hash functions are a useful tool. Below, we will just present a simplified version of privacy amplification; for more details, see, e.g., \cite{mosaics}. An improved bound on the security of the method is given in \cite{ISIT23}. 

Before we formalize the problem of privacy amplification, we recall the connection between a function $f:\mc X\times\mc S\to\mc A$ and its mosaic $M(f)=(D_\alpha)_{\alpha\in\mc A}$ on $\mc X$ with block index set $\mc S$. For each $\alpha$, we take $N_\alpha$ to be the incidence matrix of $D_\alpha$; in other words, $N_\alpha$ is a 01-matrix with rows indexed by $\mc X$ and columns by $\mc S$, and where
\[
    N_\alpha(x,s)=1\quad\text{if and only if}\quad f(x,s)=\alpha.
\]
Clearly, if $J$ denotes the all-ones matrix, then since $M(f)$ is a mosaic of incidence structures,
\begin{equation}\label{eq:mosaic_matrix}
    \sum_\alpha N_\alpha=J.
\end{equation}

The function $f$ is an $\varepsilon$-ACFU hash function if and only if 
\begin{enumerate}
    \item denoting the all-ones vector of appropriate dimension by $j$,
    \begin{equation}\label{eq:const_comp_matrixform}
        N_\alpha j=\frac{\abs{\mc S}}{\abs{\mc A}}j,
    \end{equation}
    \item and for $x\neq x'$,
    \[
        (N_\alpha N_\alpha^T)(x,x')\leq\frac{\varepsilon\abs{\mc S}}{\abs{\mc A}}.
    \]
\end{enumerate}
In particular, for any nonnegative vector $p\in\mbb R^{\mc X}$,
\begin{align}\label{eq:bil_form}
    p^TN_\alpha N_\alpha^Tp
    &\leq \frac{\abs{\mc S}}{\abs{\mc A}}p^Tp+\frac{\varepsilon\abs{\mc S}}{\abs{\mc A}}\bigl((p^Tj)^2-p^Tp\bigr)\notag\\
    &=\frac{\abs{\mc S}}{\abs{\mc A}}\bigl((1-\varepsilon)p^Tp+\varepsilon (p^Tj)^2\bigr). 
\end{align}
This inequality will be the key in the application of $\varepsilon$-ACFU hash functions to privacy amplification.

Now assume that $X$ and $Z$ are random variables on the finite alphabets $\mc X$ and $\mc Z$, respectively, with joint probability vector $p_{XZ}$. ``Privacy amplification'' means that we want to transform $X$ into another random variable $A$ whose probability distribution is close to uniform and about which an adversary observing $Z$ knows as little as possible. For this transformation, we use an $\varepsilon$-ACFU hash function $f:\mc X\times\mc S\to\mc A$ with an associated family $(N_\alpha)_{\alpha\in\mc A}$ of 01-matrices. For the second argument of $f$, we choose an input uniformly at random, which may also be known to the adversary. This setting gives us a joint probability distribution on $\mc X\times\mc Z\times\mc S\times\mc A$ for the random variables $X,Z,S,A$,
\[
	p_{XZSA}(x,z,s,\alpha)=p_{XZ}(x,z)\cdot\frac{1}{\abs{\mc S}}\cdot N_\alpha(x,s).
\]

For any $z\in\mc Z$, we define the vector $p_z\in\mbb R^{\mc X}$ by 
\[
	p_z(x)=p_{XZ}(x,z),
\]
such that $p_z^Tj=p_Z(z)$. The joint distribution of $Z$, $S$ and $A$ is
\begin{equation}\label{eq:jointdistr}
	p_{ZSA}(z,s,\alpha)
	=\frac{1}{\abs{\mc S}}\sum_xp_{XZ}(x,z)N_\alpha(x,s)=\frac{1}{\abs{\mc S}}(p_z^TN_\alpha)(s).
\end{equation}
In particular, by \eqref{eq:const_comp_matrixform}, 
\begin{equation}\label{eq:A_Z_distr}
	p_{ZA}(z,\alpha)=p_z^Tj\cdot\frac{1}{\abs{\mc A}}=p_Z(z)\cdot\frac{1}{\abs{\mc A}}.
\end{equation}
Hence $Z$ and $A$ are stochastically independent, and we even obtain a  uniform distribution for $A$ (not just an approximation). 

Before we can show that the adversary knows little about $A$, we have to define what this should mean. We impose the strong requirement that the adversary, knowing $S$ and $Z$, should not be able to distinguish any two values $\alpha,\alpha'$ of which it knows that one is the true one. In order to formalize this, we recall the definition of conditional probabilities. If $Y_1,Y_2$ are any random variables with joint probability distribution $p_{Y_1Y_2}$, then the conditional probability of $Y_1$ given $Y_2=y_2$ is defined by
\[
	p_{Y_1\vert Y_2=y_2}(y_1)=p_{Y_1\vert Y_2}(y_1\vert y_2)=\frac{p_{Y_1Y_2}(y_1,y_2)}{p_{Y_2}(y_2)}
\]
if $p_{Y_2}(y_2)>0$, otherwise it is undefined. We require that
\begin{equation}\label{eq:security}
	\max_{\alpha,\alpha'}\lVert p_{ZS\vert A=\alpha}-p_{ZS\vert A=\alpha'}\rVert_1
\end{equation}
should be sufficiently small, where $\lVert\cdot\rVert_1$ is the $\ell_1$-norm on the set of probability vectors, or equivalently, the total variation distance on the space of probability measures. (It will depend on the application what ``sufficiently'' means. See Example \ref{ex:iid_source}.) By \eqref{eq:jointdistr} and \eqref{eq:A_Z_distr},
\[
    P_{ZS\vert A}(z,s\vert \alpha)
    =\frac{\abs{\mc A}(p_z^TN_\alpha)(s)}{\abs{\mc S}}.
\]

Let $p_Zp_S$ denote the product of the distributions $p_Z$ and $p_S$, such that $p_Zp_{S}(z,s)=p_Z(z)/\abs{\mc S}$. Without loss of generality, we may assume that $p_Z$ is everywhere positive. In order to bound \eqref{eq:security}, it is sufficient to find, for every $\alpha\in\mc A$, an upper bound for 
\begin{align*}
	\lVert p_{ZS\vert A=\alpha}-p_Zp_{\mc S}\rVert_1
	&=\sum_{z,s}p_Z(z)p_S(s)\left\lvert\frac{p_{SZ\vert A=\alpha}(z,s)}{p_Z(z)p_S(s)}-1\right\rvert\\
	&\leq\left(\sum_{z,s}p_Z(z)p_S(s)\left(\frac{p_{SZ\vert A=\alpha}(z,s)}{p_Z(z)p_S(s)}-1\right)^2\right)^{1/2}\\
	&=\left(\sum_{z,s}\frac{p_{SZ\vert A=\alpha}(z,s)^2}{p_Z(z)p_S(s)}-1\right)^{1/2}\\
	&=\left(\frac{\abs{\mc A}^2}{\abs{\mc S}}\sum_{z,s}\frac{(p_z^TN_\alpha)(s)^2}{p_z^Tj}-1\right)^{1/2}.
\end{align*}
Using \eqref{eq:bil_form}, the term under the square root satisfies
\begin{align*}
	&\frac{\abs{\mc A}^2}{\abs{\mc S}}\sum_z\frac{p_z^TN_\alpha N_\alpha^Tp_z}{p_z^Tj}-1\\
	&\leq\abs{\mc A}\sum_z\frac{(1-\varepsilon)p_z^Tp_z+\varepsilon(p_z^Tj)^2}{p_z^Tj}-1\\
	&=(1-\varepsilon)\abs{\mc A}\sum_z\frac{p_z^Tp_z}{p_z^Tj}+\abs{\mc A}\varepsilon-1.
\end{align*}
By definition, the sum can be written in terms of a conditional R\'enyi 2-entropy,
\[
	\log\sum_z\frac{p_z^Tp_z}{p_z^Tj}=\log\sum_zp_Z(z)2^{-H_2(X\vert Z=z)}=-H_2(X\vert Z).
\]
(Note: There exists a different definition of conditional R\'enyi 2-entropy, see, e.g., \cite{BBCM_gen_PA}.) Hence we obtain the following result.

\begin{thm}\label{thm:PA}
    Let $f:\mc X\times\mc S\to\mc A$ be an $\varepsilon$-ACFU hash function and let $S$ be uniformly distributed on $\mc S$. Choose any $H\geq 0$. Then for any pair $(X,Z)$ of random variables independent of $S$ and satisfying $H_2(X\vert Z)\geq H$, the key $A=f(X,S)$ is uniformly distributed on $\mc A$ and satisfies
	\[
		\max_{\alpha,\alpha'\in\mc A}\lVert P_{ZS\vert A=\alpha}-P_{ZS\vert A=\alpha'}\rVert
		\leq2\left((1-\varepsilon)\abs{\mc A}2^{-H}+\abs{\mc A}\varepsilon-1\right)^{1/2}.
	\]
\end{thm}

Generally, in order to achieve good security in the theorem, we need $\varepsilon\approx1/\abs{\mc A}$. For instance, consider the following example.

\begin{ex}\label{ex:iid_source}
	Let a family $\{(X_n,Z_n):n\geq 1\}$ of pairs of random variables be given. Assume that there exists a number $H>0$ such that $H_2(X_n\vert Z_n)\geq nH$ for sufficiently large $n$. This is the case in the typical situation of secret-key generation from i.i.d.\ correlated sources where $n$ indicates the number of observed source realizations. Let $\delta>0$ and choose for every $n$ an $\varepsilon$-ACFU hash function $f_n:\mc X_n\times\mc S_n\to\mc A_n$ such that $X_n$ lives on $\mc X_n$ and with $\varepsilon\leq 1/\abs{\mc A_n}$ and $\log\abs{\mc A_n}\leq n(H-\delta)$. Then the theorem implies that \eqref{eq:security} tends to $0$ exponentially in $n$. In particular, this shows that the best-known key rate in the situation of an i.i.d.\ source is achievable even with our stronger-than-usual security measure, see, e.g., \cite[pp.~151-159]{BlochBarros}.
\end{ex}

 A related application of ACFU hash functions is to \textit{wiretap channels}, see \cite{mosaics}. In this application, an additional requirement is that the blocks $\{x:f(x,s)=\alpha\}$ have constant size. We have seen in Section \ref{sect:seed_lbs} that this poses no real restriction on the ACFU functions.
 
 For application in privacy amplification and the wiretap channel problem, there exist functions which have a smaller seed than ACFU hash functions if $\abs{\mc S}>\abs{\mc X}$, namely $\abs{\mc S}=\abs{\mc X}$, but which also achieve the best known key or channel rates in the standard settings like the i.i.d.\ source setting from the example. They are based on mosaics of near-Ramanujan graphs, i.e., edge decompositions of a complete bipartite graph with equal-sized color classes into subgraphs each of which has a very small second-largest eigenvalue \cite{BRI}. However, so far we do not know of any such mosaics whose corresponding functions are efficiently computable.

\section{Open questions}

After our extension results (Theorems \ref{thm:aff_ext} and \ref{thm:block_set_ext}), we discussed how the original function $g$ and the generated $\hat g$ or $\check g$ relate with respect to equalities in the lower bounds on the seed sizes. What remained open was whether every seed-optimal OCFU hash function can be derived from a seed-optimal OU hash function. Formulated in terms of mosaics and designs, the question is: \textit{Are the members of every mosaic of BIBDs resolvable?} In other words, is the method of Gnilke, Geferath and Pav\v cevi\'c (Corollary \ref{cor:mos_constr_primal}) essentially the only way of constructing a mosaic of BIBDs? By Corollary \ref{cor:OCFU_S_lb}, the members of a mosaics of BIBD$(v,k,\lambda)$ certainly need to satisfy the necessary condition $b\geq v+r-1$ for resolvable designs. 

If this question can be answered in the positive, then this also implies that dually, any $\varepsilon$-ACFU hash function with equality in \eqref{eq:seed_size_lb_variance} is the point extension of an $\varepsilon$-ASU hash function satisfying equality in \eqref{eq:ASU_seed_eq}. Another consequence would be that the sum of a mosaic of BIBDs is doubly-resolvable \cite[Remark I.5.16]{BJL_book}.

More generally, a similar question can be posed about the structure of mosaics which neither in their ``primal'' nor their dual version consist of BIBDs. In terms of ACFU hash functions, this could in particular clarify the relation between seed-optimal ACFU hash functions with equality in \eqref{eq:seed_size_lb_simple} and seed-optimal ASU hash functions with equality in Lemma \ref{lem:ASU_seed_simple}.

\section*{Acknowledgments}

M. Wiese and H. Boche were supported by the German Federal Ministry of Education and Research (BMBF) within the programme ``Souverän. Digital. Vernetzt.'', project 6G-life, grant 16KISK002, and within the project NewCom under grant 16KIS1003K.  H. Boche was additionally supported
in part by BMBF through the project ``Quantum Token Theory and Applications -- Q.TOK'' under grant 16KISQ037K.

\end{document}